\documentclass[12]{amsart}
\usepackage{graphics}
\usepackage{amssymb}
\usepackage{hyperref}

\usepackage{pdfpages}

\hypersetup{
      colorlinks=true,
      linkcolor=blue,
      filecolor=magenta,
      urlcolor=cyan,
}

\newtheorem{theorem}{Theorem}[section]
\newtheorem{lemma}[theorem]{Lemma}

\theoremstyle{definition}
\newtheorem{definition}[theorem]{Definition}
\theoremstyle{corollary}
\newtheorem{corollary}[theorem]{Corollary}
\theoremstyle{example}
\newtheorem{example}[theorem]{Example}
\theoremstyle{note}
\newtheorem{note}[theorem]{Note}
\theoremstyle{remark}
\newtheorem{remark}[theorem]{Remark}
\numberwithin{equation}{section}
\newcommand{\bs}{\backslash}

\newcommand{\imply}{\longrightarrow}

\newcommand{\X}[1]{spec \ {#1}}
\begin{document}

\title{Study of spectrum of certain subrings of a commutative ring with identity}

\author{Biswajit Mitra }

\address{Department of Mathematics. The University of Burdwan, Burdwan 713104, West Bengal, India }

\email{b1mitra@math.buruniv.ac.in}

\author{Debojyoti Chowdhury}

\address{Department of Mathematics. The University of Burdwan, Burdwan 713104, West Bengal, India}
\email{sankha.sxc@gmail.com}

\author{Sanjib Das }

\address{Department of Mathematics. The University of Burdwan, Burdwan 713104, West Bengal, India}
\email{ruin.sanjibdas893@gmail.com}
\subjclass[2010]{Primary 13AXX; Secondary 13A15, 13A99, 54H10}
\keywords{Commutative ring with unity, pm rings, rings of continuous functions, spectrum, maximal spectrum.}

\begin{abstract}
By a ring we always mean a commutative ring with identity. It is well known that maximal spectrum of $C(X)$, $C^*(X)$ and any intermediate subrings between $C(X)$ and $C^* (X)$ are homeomorphic and homeomorphic with $\beta X$, the Stone-$\check{C}$ech compactification of $X$. In this paper we generalized these results to an arbitrary ring by introducing a notion of dense subring. We proved that if $A$ is completely normal and dense subring of $B$, then maximal spectrum of $A$ and $B$ are homeomorphic and hence maximal spectrum of all intermediate subrings between $A$ and $B$ where $A$ is dense, are homeomorphic. We also proved that $A$ is dense subring of $B$ if and only if spectrum of $B$ is densely embedded in spectrum of $A$ and have further shown that if $A$ is dense subring of $B$, any minimal prime ideal of $A$ is precisely of the form $Q\cap A$ for some unique minimal prime ideal $Q$ of $B$. As a consequence, we concluded that if $A$ is dense in $B$,  minimal spectrum of $A$ and that of $B$ are homeomorphic. We also studied different properties of dense subrings of a ring.
\end{abstract}

\maketitle

\section{Introduction} Throughout the paper, term ring denotes a commutative ring with identity, to be denoted usually by $1$, all subrings of a given ring contain the identity of whole ring and all ring homomorphisms preserve identities. In commutative ring theory one has a functorial construction that corresponds a ring with its prime spectrum which is the space of all prime ideals of the ring with Zarisky topology. One further associates a ring with the subspace of prime spectrum, called maximal spectrum, consisting of all maximal ideals with the Zarisky topology restricted on it. However this construction is not in general functorial. But within the class of pm-rings, that is the rings where every prime ideal can be extended to unique maximal ideal, the construction turns out to be functorial as follows from \cite{MO71}, theorem 1.5.  Now for any space $X$, let $C(X)$ denotes the family of all real-valued continuous functions on $X$ and $C^*(X)$, the subring of all bounded real-valued continuous functions. Any subring of $C(X)$ containing $C^*(X)$ is referred as intermediate subring. Maximal spectrum of $C(X)$ (respectively $C^*(X)$) are usually referred as structure space of $C(X)$ (respectively $C^* (X)$). From \cite{GJ60}, it is well known that the structure space of $C(X)$ and structure space of $C^* (X)$ are homeomorphic and homeomorphic with $\beta X$, the Stone-$\check{C} $ech compactification of $X$. More generally structure space of any two intermediate subrings of $C(X)$ are homeomorphic. However relation among spectrum of $C(X)$ and that of $C^*(X)$ has been explored by Macro and Orsatti in \cite{MO71}. They proved that spectrum of $C^* (X)$ is embedded in spectrum of $C(X)$. The way they had proved the result, can immediately be generalized if $B$ is a ring of fraction of $A$.  In this paper we proved, infact in more general framework, that spectrum of $C(X)$ is not only embedded in that of $C^* (X)$ but also as a dense subspace and further observed that complete normality of spectrum of $C^*(X)$ is directly responsible for spectrum of $C(X)$ to be completely normal and even more, causing their structure spaces homeomorphic. We, infact, don't need to directly involve in manipulating structure of maximal ideals of $C(X)$ and $C^*(X)$ to establish the homeomorphism only. However it is required whence to show their structure spaces homeomorphic with $\beta X$. In this paper we validated the extent of these results to an arbitrary ring via introducing a notion, referred as `dense subring'.

In section 3, we introduced the notion of dense subring of a ring and studied different properties of such subrings. In section 4, we first established that a subring $A$ is dense in $B$ if and only if spectrum of $B$ is densely embedded in spectrum of $A$. We defined a subring $A$  of a ring $B$ to be weak completely normal with respect to  $B$  if  for any two distinct maximal ideal $M \neq M'$,  whenever $M \cap A$ and $M' \cap A$ are non comparable, $cl_{\X{A}} \{M \cap A\} \cap cl_{\X{A}} \{M' \cap A\} = \emptyset$. A ring $A$ is called weak completely normal if $A$ is weak completely normal with respect to any ring containing it. We have shown that if $A$ is pm-subring of $B$, then maximal spectrum of $B$ is homeomorphic with that of $A$ if and only if  $A$ is dense and weak completely normal subring with respect to $B$. As we have proved that a completely normal ring (i.e. $\X{A}$ is completely normal), is weak completely normal and $C^* (X)$ is dense and completely normal subring of $C(X)$, maximal spectrum or structure space of $C^*(X)$, is homeomorphic with that of $C(X)$.     \\

We further proved that maximal spectrum of any two subrings between $A$ and $B$, to be referred as intermediate subrings, are homeomorphic, provided $A$ is dense and weak completely normal, with respect to those intermediate subrings and justified that $C^*(X)$ is dense and completely normal with respect to any intermediate subrings of $C(X)$ and hence their structure spaces are homeomorphic with that of $C(X)$. Finally we have shown that if $A$ is dense in $B$, then space of minimal prime ideal of $A$ under Zarisky topology is homeomophic with that of $B$.

\section{Preliminaries}
Let $A$ be a ring. $\X{A}$ denotes the family of all prime ideals with Zariski topology. Then $\X{A}$ is a compact $T_0$ space [\cite{AM},Chapter 1, Excercise 15]. $\{X^A_f: f \in A\}$ where $X^A_f = \{P \in \X{A}: f \notin P\}$, is the family of all basic open sets in $\X{A}$. Let $V^A_f = \X{A} \bs X^A_f$ is the basic closed sets for each $f\in A$ in $\X{A}$.  Sometimes we write $X_f$ instead of $ X_f^A$, if there is no confusion about the underlying ring $  A$. Let $max \ (A) $ denotes the family of all maximal ideals in $A$. Clearly $max \ (A) \subseteq \X{A}$. The $max \ (A)$ with restricted Zarisky topolgy is called maximal spectrum of $A$. It is well known [\cite{AM},Chapter 1, Excercise 15], that, maximal spectrum is compact $T_1$ space. However it is not in general $T_2$. It is $T_2$ if and only if the ring $A$ satisfies the following condition  [\cite{GJ60}, Chapter 7. Excercise 7M].

For any two distinct maxcimal ideals $M, M'$ of $A$, there exist $a\notin M, a'\notin M'$ so that $aa' \in jac \ A$ where $jac \ A$ is the Jacobson radical of $A$, i.e, it is the intesection of all maximal ideals of $A$.

For each $M \in max \ (A)$, let $\mathcal{P}_M$ denotes the family of all prime ideals contained in $M$ and $O_M = \cap \mathcal{P}_M$. $O_M$ is in general need not be prime. An alternate description of $O_M$ is given in the paper of Marco and Orsatti in \cite{MO71}.

Further $O_M$ can be expressed as intersection of minimal prime ideal follows from [\cite{AM}, Chapter I, Excercise 20].

In general spectrum of a ring is very bad with respect to separation axioms, however with the class of pm-rings it behaves nicely. A pm-ring is a ring in which every prime ideal is contained in a unique maximal ideal. Let $A$ be a pm-ring and $\mu_A : spec \ A \to max (A)$ be the map defined by $\mu_A (P)$ to be the unique maximal ideal containing $P$.  Marco and Orsatti described the following characterization of pm-rings [\cite{MO71}, Theorem 1.2]

\begin{theorem} \label{2.1}
Let $A$ be a ring. The followings are equivalent.

(a) $A$ is pm-ring

(b) $max \ (A)$ is a retract of $spec \ A$ under the retraction $\mu_A$.

(c) For each maximal ideal $M$ of $A$, $M$ is the unique maximal ideal containing $O_M$ (i.e $\mathcal{P}_M$ is closed in $\mathcal{P})$.

(d) $spec \ A$ is normal.

Furthermore, if (a) holds, the map $\mu_A$ is the unique retraction of $spec \ A$ onto $max(A)$ and $ max (A)$ is $T_2$.
\end{theorem}

The following lemma [\cite{MO71}, lemma 1.4] has been very useful in this paper. Infact we have completely characterized the following lemma in more general framework in the next section.

\begin{lemma}\label{2.2}
Let $A$ a be subring of a ring $B$. For every $Q \in spec \ A$, there exists $P \in spec \ B$ such that  $P\cap A \subseteq Q$.
\end{lemma}

The following theorem [\cite{MO71}, Theorem 1.6] is the key result, applied in this paper to investigate homeomorphism between maximal spectrum of $A$ and $B$.

\begin{theorem}\label{2.3}
Let $B$ be a ring, $A$ a subring of $B$. Assume that $A$ is pm-ring. Then the map $\lambda: max(B) \to max (A)$ which sends every maximal ideal $M$ of $B$  into the unique maximal ideal of $A$ containing $ M \cap A$ is a continuous closed map from $\max (B)$ onto $max ( A)$; $\lambda$ is a homeomorphism  if and only if $M, M' \in max \ B$ and $M \neq M'$ imply $M \cap A + M'\cap A = A$.
\end{theorem}

Schwartz and Tressl in [\cite{ST10}, theorem 4.3,  theorem 5.1] have shown  lists of characterizations of spectrum of a ring to be normal,i.e, that of pm-rings. Moreover in the same paper, they characterized completely normal spectrum of a ring. For simplicity, we shall refer a ring to be completely normal if their respective spectrum is so.

Infact in [\cite{ST10}, corollary 6.4], they proved the following result which assures that homomorphic image of a completely normal ring is also completely normal.

\begin{theorem}\label{2.4}
Let $A$ be a ring. $\X{A}$ is completely normal if and only if for all $s,a \in A$, there are $x,x' \in A$ and $k \in \mathbb{N}$ such that

$(s^k - xsa)(s^k - x'(s^2 - sa)) = 0$
\end{theorem}

We have used the following theorem in [\cite{AM}, Chapter 1, Excercise 21] to conclude corollary 4.2.

\begin{theorem}\label{2.5}
Let $A$ and $B$ be two rings and $f : A \to B$ is a ring homomorphism. Then the following are true.

a. $f^*: spec \ B \to spec \ A$, defined by $f^*(P) = f^{-1} (P)$ is continuous.

b. $f^* (spec \ (B))$ is dense in $spec \ (A)$ if $f$ is injective.
\end{theorem}

In our framework, as $A$ is a subring of $B$, $f$ is here the inclusion map $i$ and hence $i^*(\X{B})$ is dense in $\X{A}$. Another simple point to note is $i^{-1}(P) = P \cap A$.

\section{Introduction of dense subrings and few of its properties}

 In this section we shall introduce the definition of a dense subring of a ring. But in prior to it,as promised, we first proved the vital lemma \ref{2.2},as stated, in the last section in more general framework.

 \begin{theorem}
  $A$ and $B$ be two rings. Let $f : A \to B$ be a ring homomorphism. $f^* : spec \ B \to spec \ B $ is the map defined in theorem \ref{2.5}. Then the followings are equivalent.

\begin{enumerate}
 \item kernel of $f$ is contained in nilradical of $A$
 \item $f^* (spec \ B )$ is dense in $spec \ A$
 \item For each prime ideal $P$ os $A$, there exists a prime ideal $Q$ in $B$ such that $f^*(Q) \subseteq P$
\end{enumerate}

 \end{theorem}

  Now in our set up, $A$ is a subring of $B$, with $f$ the inclusion map $i$ and $i^* (P) = P \cap A$ and furthermore kernel of $i$ is $\{0\}$ and so satisfies (1) and hence the lemma \ref{2.2} follows directly from the above theorem. Moreover $i^* (spec \ B) $ is dense in $spec \ A$. $i^*$ is also continuous as follows from theorem \ref{2.5}.

 \begin{proof}
  (1) if and only if (2) follows from chapter 1, excercise 21 of \cite{AM}. \\

  $(2) \Rightarrow (3)$: First we note that $P \in X_a^A$ if and only if $a \notin P$. So for each $a \notin P$, there exists a prime ideal $Q_a$ in $B$ such that $f^* (Q)$ is in $X_a^A$. Let $S = A \backslash P$ and $I =  \bigcap_{a \in S} Q_p$. Then $I \cap f(S) = \emptyset$ and $f(S)$ is multiplicatively closed subset of $B$. So there exists a prime ideal $Q$ in $B$ such that $Q$ contains $I$ and $Q \cap S = \emptyset$. This follows that $f^* (Q) \cap (A \backslash P) = \emptyset$. So $f^* (Q) \subseteq P$. \\

  $(3) \Rightarrow (2)$: Let $P \in spec \ A$. So by hypothesis, there exists a $Q$ in $ spec \ B$, such that $f^* (Q) \subseteq P$. So any basic open set containing $P$ must contain $F^* (Q)$. So $f^* (spec \ B )$ is dense in $spec \ A$.\\
 \end{proof}

 Now we shall give formal definition of dense subring of a ring which is the most essential and principal part of this paper.

 \begin{definition}\label{3.1} Let $A$ be a subring of $B$ containing $1$. $A$ is called dense subring of $B$ if for any ideal $I$ of $B$,and, for any $b \notin rad(I)$, there exists $a \notin rad(I)$, such that $ab \in A$.
 \end{definition}

Few examples of dense subrings are given below.

\begin{example}

1. The integers ring $ \mathbb{Z}$ is dense in the rational ring $\mathbb{Q}$.

2. For any space $X$, $C^* (X)$, infact any intermediate subrings, are dense  in $C(X)$. More generally, if $B$ is a ring of fraction of $A$, then $A$ is dense in $B$.

3. $\mathbb{Q}$ is also dense in $\mathbb{R}$. Infact we have the following general result

\end{example}

\begin{theorem} \label{3.1a}
If $A$ is a subring of $B$ and $B$ is a subring of $C$ and further $A$ is dense in $C$, then $B$ is also dense in $C$.
\end{theorem}

\begin{proof}
  Since $A$ is dense in $C$, for any ideal $I$ of $C$ and for any $a \notin rad (I)$, there exists a $b \notin rad (I)$, $ab \in A$. Hence $ab \in B$. So $B$ is dense in $C$
\end{proof}

However we don't know whether the following result is true or even we could not crack any counter example.  If $A$ is dense in $B$, thes whether $A$ is dense in any intermediate subring of $B$. By the way, any subring lying between $A$ and $B$, henceforth, will be referred as intermediate subring of $B$.

%********************************************************************************************************************************************
%Let for any subring $A$ of $B$, we denote $\bar{A} = \{ b \in B: \exists a \in B \ with \ a \neq 0, ba \in A\}$. Clearly $A \subseteq \bar{A}$ (infact %for any $a \in A$, choose $1 \in A$).

%\begin{theorem}
% $A$ is dense in $B$ if and only if $\bar{A} = B \bs \script{N} (A)$. Hence %if $\script{N}(A)  = 0$, then $A$ is dense in $B$ if and only if $\bar{A} = %B $
%\end{theorem}

%\begin{theorem}
%For any two subrings $A \subseteq B$, $\bar{A}\subseteq \bar{B}$.
%\end{theorem}

%\begin{proof}
%Trivially follows from the definition.
%\end{proof}

%\begin{theorem}
%For any subring $A$, $\bar{\bar{A}} = \bar{A}$
%\end{theorem}

%\begin{proof}
%Let $b \in \bar{\bar{A}}$, there exists $a \in B a \neq 0$, $ab \in \bar{A}%%$. So there exists $c \in B, c \neq 0$, $(ab)c \in A$, that is, $a(bc)\in A%$. $bc \neq 0$ as \mathcal{N}(A) = %0$
%\end{proof}
%*********************************************************************************************************************************************

However we have few more properties of denseness below.

\begin{theorem}\label{3.2}

If $A$ is dense in $B$ and $B$ is dense in $C$, then $A$ is dense in $C$.
\end{theorem}

\begin{proof} Let $I$ be an ideal of $C$ and $b \notin rad(I)$. So there exists a prime ideal $P$ so that $b \notin P$. Now for any prime ideal $P$, $rad{P} = P$. So there exists $b \notin P$ so that $ab \in B$. Clearly $ab \notin P\cap B$. Again $P\cap A$ is  a prime ideal of $B$. So there exists a $c \notin P\cap B$ such that $cab \in A$. Clearly $ca \notin P$ and hence not in $rad (I)$ but $(ca)b \in A$.
\end{proof}

Unfortunately we don't know whether $A$ is dense in any intermediate subring of $B$ whence $A$ is dense in $B$. We could not crack any counter example in this regard. However we move forward to the following important theorem of this paper.

\begin{theorem} \label{3.3} If $A$ is dense in $B$, then the map $i^* : spec B \to spec A$ is one--one.
\end{theorem}

\begin{proof} Let $P \neq Q, P, Q \in B$. Suppose $a \in P \bs Q$. There exists $b \notin Q$ such that $ab \in A$. So $ab \in P \cap A \bs Q\cap A$. So  $i^*(P) = P\cap A \neq Q\cap A = i^* (Q)$.
\end{proof}

\begin{remark} \label{3.4} In the above proof, mere interchanging the role of $P$ and $Q$, we can conclude that if two prime ideals of $B$ are non-comparable, their $i^*$ images are still non-comparable, provided $A$ is dense in $B$.
\end{remark}

\begin{lemma}\label{3.5} If $A$ is dense in $B$, $P,Q$ are two prime ideals of $B$, then, $P \cap A \subseteq Q \cap A$ implies that $P \subseteq Q$.
\end{lemma}

\begin{proof}
Suppose $a \notin Q$. Due to density, there exists a $ b \notin Q$ such that $ab \in A$. So $ab \notin Q \cap B$. This implies $ab \notin P\cap A$. So $ab \notin P$. Hence $a \notin P$. So $ P \subseteq Q$.
\end{proof}

The following theorem depicts a structure of minimal prime ideal of dense subrings of a ring.

\begin{theorem} \label{3.6}
If $A$ is dense in $B$, any minimal prime ideal of $A$ is precisely of the form $Q \cap B$, for some unique minimal prime ideal $Q$ of $B$.
\end{theorem}

\begin{proof} Let $P$ be a minimal prime ideal of $A$. By [lemma 1.4, \cite{MO71}, there exists a prime ideal $Q$ in $B$ such that $Q \cap A \subseteq P$. So $Q \cap A = P$. Suppose now $Z$ be a prime ideal of $B$ such that $Z \subseteq Q$. Then $Z \cap A \subseteq Q \cap A = P$. Due to minimality of $P$, $ Z \cap A = P = Q\cap A$. Thus $i^* (Z) = i^* (Q)$. As $i^*$ is injective, (by theorem \ref{3.3}), $Z =  Q$. So $Q$ is minimal. Uniqueness of $Q$ again follows from injectivity of $i^*$.  Conversely suppose $Q$ is minimal. Let P be a prime ideal of $A$ contained in $ Q \cap A$. Then by lemma 1.4 of \cite{MO71}, there exists a prime ideal $Z$ so that $Z \cap A \subseteq P$. By lemma \ref{3.5}, $Z \subseteq Q$. Hence $Z = Q$, due to minimality of $Q$. So $ P = Q \cap A$. Thus $Q\cap A$ is minimal.
\end{proof}

\begin{theorem}\label{3.7}
If $A$ is dense in $B$ and $B$ is a pm-ring, then for any two distinct maximal ideals $M \ and \ M'$, $(M\cap A) \cap (M' \cap A)$ does not contain any prime ideal of $A$,
\end{theorem}

\begin{proof}
Let $P$ be a prime ideal of $A$ such that $P \subset (M \cap A) \cap (M' \cap A) = M \cap M' \cap A$. By lemma 1.4 of \cite{MO71}, there exists a prime ideal $Q$ of $B$ such that $Q \cap A \subseteq P \subseteq M\cap M' \cap A$. By lemma \ref{3.5}, $Q \subseteq M \cap M'$. Then $M = M'$ as $B$ is a pm-ring, a contradiction.
\end{proof}

The following theorem is mere a observation and has no follow-up in this paper.

\begin{theorem}
If $A$ is dense in $B$ and $u \in A$ is a unit of $B$, then for any inverse $v$ of $u$ in $B$, the collection $C^v_u = \{r \in B: rv \in A\}$ is a non-trivial ideal of $A$.
\end{theorem}

\begin{proof}
 As $rv \in A, u \in A, (rv)u = r \in A$. $C^v_u $ is a subset of $A$. Checking of ideal is trivial. And non-triviality follows from denseness of $A$.
\end{proof}

\section{Effect of denseness on Spectrum and maximal spectrum of a ring}

In this section we shall explore the relation between spectrum of rings $A$ and $B$ if $A$ is dense in $B$.

\begin{theorem} \label{4.1}
Let $A$ and $B$ be two commutative rings with 1, with $A$ subring of $B$. Then $A$ is dense in $B$ if and only if $i^*: \X{B} \to \X{A}$ is one and open (Here open means that if $U$ is open in $\X{B}$, $i^*(U)$ is open in $i^*(\X{B})$.)
\end{theorem}

\begin{proof}
Suppose $A$ is a dense subring of $B$. By thorem \ref{3.3}, $i^*$ is one-one. It is left to show that $i^*$ is open.

Let $X_f^B$ be a basic open set in $\X{B}$, where $f \in B$. Let $Q \in i^* (X_f^B)$. As $Q \in i^* (\X{B})$, there exists $P \in \X{B}$ such that $i^* (P) = Q$. That is $Q = P\cap A$.

As $i^*$ is one-one, $P \in X_f^A$. So $f \notin P$. But $P = rad P$.

Since $A$ is dense in $B$, there exists $g \notin P$, $fg \in A$. As $f,g \notin A$, $fg \notin P\cap A$. So $i^* (P) = P \cap A \in X_{fg}^A$, that is, $Q \in X_{fg}^A$. As $X_{fg}^A$ is a basic open set in $\X{A}$, $X_{fg}^A \cap i^* (\X{B})$ is a basic open set containing $Q$.

We claim that $X_{fg}^A \cap i^* (\X{B})\subseteq i^* (X_f^B)$ and this will prove that $i^* (X_f^B)$ is an open subset of $i^* (\X{B})$.

For that, let $Q' \in X_{fg}^A \cap i^* (\X{B})$. This implies that there exists $Z \in \X{B}$, such that $Q'= Z\cap A$. Now $fg \notin Q' = Z \cap A$. Since $fg \in A$, $fg \notin Z$. So $f,g \notin Z$. This implies that $Z \in X_f^B$. Hence $i^* (Z) \in i^* (X_f^B)$. So $ Q' \in i^* (X_f^B)$.

Conversely, Suppose $i^*$ is one-one and open. We show that $A$ is dense subring of $B$.

Let $I$ be an ideal of $B$ and $b \notin rad \ I$. So there exists a prime ideal $P_0 \supseteq I$, $b \notin P_0$. This implies that $P_0 \in X_b^B$.
Let $Q_0 \in i^* (P_0) = P_0 \cap A \in i^* (X_b^B)$. Since $i^*$ is open, $i^*(X_b^B)$ is open in $i^*(spec \ B)$. Hence there exists $a \in A$ such that

$Q_0 \in X_a^A \cap i^* (\X(B) \subseteq  i^*(X_b^B)$. ...  \  \  (1)

Now we claim that for any prime ideal $P$ of $B$,

$b \in P \imply a \in P$.

For that suppose $a \notin P$. So $a \notin P \cap A$. Thus $P \cap A \in X_a^A$. Also $P \cap A = i^* (P) \in i^*(\X{B})$.

So $i^*(P) \in X_a^A \cap i^*(\X{B}) \subseteq i^* (X_b^B)$ \ ... by (1)

As $i^*$ is one-one, $P \in X_b^B$. Hence $b \notin P$. So we proved our claim.

But the above claim implies that $a \in rad \ (b)$, where $(b)$ is the ideal generated by $b$. So there exists a natural number $n$, $a^n \in (b)$. Thus $a^n = ub$ for some $u \in B$. Clearly $ub \in A$. Now $u \notin P_0$ because if $u \in P_0$, then $a^n$ and hence $a$ is in $P_0$. So $a \in P_0 \cap A = Q_0$ which is a contradiction as by (1), $Q_0 \in X_a^A$. So $u \notin P_0$ and hence $ u \notin rad I$.

So $A$ is dense in $B$.
\end{proof}

We have the following immediate corollary.

\begin{corollary}\label{4.2}
$A$ is dense in $B$ if and only if $\X{B}$ is densely embedded in $\X{A}$ via the embedding $i^*$.
\end{corollary}

\begin{proof}
According to Chapter 1, Excercises 21 (i and iv)of \cite{AM}, $i^*$ is continuous and $i^* (\X{B})$ is dense in $\X{A}$. Since $A$ is dense in $B$, $i^*$ is one-one and open. So $i^*$ is embedding. Hence $ \X{B}$ is densely embedded in $\X{A}$. Conversely if $i^*$ is an embedding, then in particular $i^*$ is open and one-one. So by the above theorem \ref{4.1}, $A$ is dense in $B$.
\end{proof}

\begin{corollary}\label{4.3}
Let $A$ be dense in B and $M$ be a maximal ideal of $B$. Then there does not exists any prime ideal $P$ of $B$ so that $ M \cap A \subsetneqq P \cap A$.
\end{corollary}

\begin{proof} Since $A$ is dense in $B$, the map $i^* : \X{B} \to \X{A}$ is an embedding. Hence $\{i^*(M)\}$ is closed in $i^* (\X{B})$ as $\{M\}$ is closed in $\X{B}$. For any $P \in \X{B}$, $M\cap A \subsetneqq P \cap A$, then $P \cap A \in cl_{i^*(\X{B})} \{M\cap A\}$. This is not possible as $\{i^*(M)\}$ is closed in $i^* (\X{B})$.
\end{proof}

We shall investigate now the effect of denseness on maximal spectrum of a ring. In \cite{MO71}, Theorem 1.6 tells that if $A$ is a pm-ring then $Max  (A) $ is homeomorphic with $Max (B)$ with respect to the map $\lambda$ which sends $M$ to the unique maximal ideal of $A$ containing $M \cap A$ if and only if for any two distinct maximal ideals $M,M'$ of $B$, $M \cap A + M \cap A = A$.

We begin with the following definition.

\begin{definition}
A subring $A$  of a ring $B$  is called weak completely normal with respect to  $B$  (henceforth referred as weak completely normal subring of $B$) if  for any two distinct maximal ideal $M \neq M'$,  whenever $M \cap A$ and $M' \cap A$ are non comparable, $cl_{\X{A}} \{M \cap A\} \cap cl_{\X{A}} \{M' \cap A\} = \emptyset$. A ring $A$ is called weak completely normal if $A$ is weak completely normal subring of any ring containing it.
\end{definition}

\begin{theorem} A ring $A$ is weak completely normal if and only if for any two non-comparable prime ideals $P$ and $P'$ of $A$, $cl_{\X{A}} \{P\} \cap cl_{\X{A}} \{P'\} = \emptyset$.
\end{theorem}

\begin{proof} Sufficient part is trivial. Conversely suppose $A$ is weak completely normal. Let $P$ and $P'$  be two non-comparable prime ideals of $A$. Take $D = (A \bs P) \cap (A \bs P')$. Then $A$ is weak completely normal subring of the quotient ring $D^{-1} A$. Then  $D^{-1} P$ and  $D^{-1} P'$ are non-comparable prime ideal of  $D^{-1} A$. Then $D^{-1} P$ and  $D^{-1} P'$  are contained in two maximal ideals $M$ and $M'$ of  $D^{-1} A$. It is clear that $P \subseteq M\cap A$ and $P' \subseteq M'\cap A$. If $M\cap A$ and $ M' \cap A$ are comparable, then $P\cup P' \subseteq M\cap A$ or $M'\cap A$. Now for any $x \notin P\cup P'$, $ x \in D$, and hence is a unit of $B$. So $x$ can not be a member of both $M$ and $M'$ and hence not a member of both $M \cap A$ and $m'\cap A$.  Thus $P\cup P' = M\cap A$ or $M'\cap A$ and in either cases, $P \cup P'$ turns out to be a prime ideal of $A$ which is not true as $P$ and $P'$ are non-comparable. So $M\cap A$ and $M'\cap A$ are non comparable in $A$. This follows that $cl_{\X{A}} \{M \cap A\} \cap cl_{\X{A}} \{M' \cap A\} = \emptyset$. As  $P \subseteq M\cap A$ and $P' \subseteq M'\cap A$, $cl_{\X{A}} \{P\} \cap cl_{\X{A}} \{P'\} = \emptyset$.
\end{proof}

From above theorem it follows that weak completely normals are precisely those rings, where any two non-comparable prime ideals are not contained in a maximal ideal . We have already observed that every completely normal ring is pm ring, converse may not be true.  However the converse is trivially true if we assume weak completely normal. So we have the following theorem.

\begin{theorem} Every weak completely normal and pm-ring is completely normal.

\end{theorem}

\begin{proof}  If $A$ is weak completely normal, any two non-comparable prime ideals are contained in distinct maximal ideals. As $A$ is pm, for any prime ideal $P$, all prime ideals containing $P$ is contained in unique maximal ideal. So all prime ideals containing $P$ is comparable to each other. Hence $A$ is completely normal.
\end{proof}

\begin{lemma}\label{4.4}
If $A$ is a subring of $B$ such that for any two distinct maximal ideal $M, M'$ of $B$, $M \cap A + M \cap A = A$, then for any two distinct maximal ideal $M \neq M'$,  $M \cap A$ and $M' \cap A$ are non comparable and $A$ is weak completely normal  subring of $B$
\end{lemma}

\begin{proof}
 Let $cl_{\X{A}} \{M \cap A\} \cap cl_{\X{A}} \{M' \cap A\} \neq \emptyset$. Choose $N \in cl_{\X{A}} \{M \cap A\} \cap cl_{\X{A}} \{M' \cap A\} $. Then $M \cap A, M'\cap A \subseteq N$. By hypothesis, $N=A$, a contradiction.
\end{proof}

In the next theorem, we shall show that the converse is true if $A$ is dense in  $B$.

\begin{theorem}\label{4.5}
Let $A$ be a  dense, weak completely normal subring of $B$. \\

Then for any two distinct maximal ideal $M, M'$ of $B$, $M \cap A + M \cap A = A$.
\end{theorem}

\begin{proof}
As $A$ is dense in $B$, for any two distinct maximal ideal $M, M'$ of $B$, $i^*(M)$ and $i^* (M')$ are non comparable. Now by hypothesis,$cl_{\X{A}} \{M \cap A\} \cap cl_{\X{A}} \{M' \cap A\} = \emptyset$. If $M \cap A + M' \cap A \neq A$, there exists a prime ideal $P$ of $A$ such that $M \cap A + M' \cap A \subseteq P$. This implies that both $M \cap A$ and $M' \cap A$ are contained in $P$. Hence $P \in  cl_{\X{A}} \{M \cap A\} \cap cl_{\X{A}} \{M' \cap A\}$, which is a contradiction.
\end{proof}

In particular both $C(X)$ and $C^* (X)$ are completely normal and weak completely normal. However complete normality of $C(X)$ or of any intermediate subrings follows more generally from the following theorem. It is to be noted that both complete normality

\begin{theorem}\label{4.7}
If $A$ is completely normal and dense subring of $B$, then $B$ is also completely normal and hence pm.
\end{theorem}

\begin{proof}
As $A$ is dense in $B$, $spec \ B$ is densely embedded in $spec \ A$. As complete normality is hereditary, $spec \ B$ is completely normal. Every completely normal space is normal.  Hence $B$ is also pm-ring by theorem \ref{2.1}.
\end{proof}

\begin{theorem}\label{4.8}
If $A$ is pm, dense and weak completely normal subring $B$, then $Max (A)$ is homeomorphic with $Max (B)$.
\end{theorem}

\begin{proof}
Since $A$ is a pm, dense and weak completely normal subring of $B$, by theorem \ref{4.5},  for any two distinct maximal ideal $M, M'$ of $B$, $M \cap A + M \cap A = A$. Then by theorem \ref{2.3},  $Max (A)$ is homeomorphic with $Max (B)$.
\end{proof}

Now if $A$ is completely normal, then $A$ is weak completely normal and pm. Thus we have the following immediate corollary.

\begin{corollary}\label{4.9a} If $A$ is dense and completely normal subring of $B$, then then $Max (A)$ is homeomorphic with $Max (B)$.
\end{corollary}

Now from the above corollary \ref{4.9a}, it directly follows that as $C^*(X)$ is completely normal and dense in $C(X)$, their structure spaces must be homeomorphic. The following corollary further takes care of the structure spaces of intermediate subrings of $C(X)$ also.

\begin{corollary} \label{4.9}
 Let $A$ be a completely normal and dense subring of $B$. Then maximal spectrum of all intermediate subrings between $A$ and $B$, where $A$ is dense, are homeomorphic.
\end{corollary}

\begin{proof}
Let $C$ be intermediate subring of $B$ and $A$ is dense in $C$ . As is completely normal, $A$ is pm and completely normal subring of $C$. By theorem \ref{4.8} $max (C)$ and $max (A)$ are homeomorphic with $max (A)$. Hence maximal spectrum of all intermediate subrings between $A$ and $B$, where $A$ is dense, are homeomorphic.
\end{proof}

This is a reason behind the structure space of intermediate subrings of $C(X)$ to be homeomorphic. Infact $C^*(X)$ is dense in any intermediate subring of $C(X)$. Let $A$ be an intermediate subring of $C(X)$ and let $I$ be an ideal of $A$ such that $f \notin rad (I)$. That means there exists a prime ideal $P$ in $A$ such that $f \notin P$. Now $1+f^2$ is in $A$. Clearly $\frac{1}{1+f^2}$ is in $C^*(X)$, hence is in $A$ also. Thus $\frac{f}{1+f^2} \in A$. Now $\frac{f}{1+f^2}$ is not in $P$. If it would be so, then $f \in P$, as $1+f^2$ is a unit of $A$. This would lead to a contradiction. Hence $\frac{f}{1+f^2}$ is not in $rad (I)$ also. Now $f. \frac{f}{1+f^2} $ is in $C^*(X)$. So $C^*(X)$ is dense in $A$. Moreover $C^*(X)$ is completely normal. The rest follows from the corollary \ref{4.9} above.

Mere denseness of a subring of a ring does not ensure the existence of homeomorphism between their maximal spectrums but we could not provide any counter-example in support of the theorem \ref{4.8}. However Mere densenss implies the following theorem. In prior to that we may recall that the space of minimal prime ideal of a ring $A$ with Zarisky topology is called the minmal spectrum of $A$ and is denoted as $min (A)$.

\begin{theorem}
If $A$ is a dense subring of $B$, then $min (A)$ is homeomorphic with $min (B)$.
\end{theorem}

\begin{proof}
 The proof follows from the above theorem \ref{3.6}. We therefore define a map $\varTheta : min(B) \to \min (A)$ by $\varTheta (P) = P \cap A$. We shall show that $\varTheta$ is a homeomorphism. That it is bijective follows from theorem \ref{3.6}.

 \

 \textbf{Continuity of $\varTheta$:} Let $X_a^A$ be a basic open set of $min (A)$. We here keep the same notation for basic open sets in $min (A)$. So  $\varTheta^{-1} (X_a^A) = \{ Q \in min (B): Q \cap A \in X_a^A\} = \{Q \in \min (B): a \notin Q\} = X_a^B$. \\

 \textbf{Openness of $\varTheta$:} Let $X_a^B$ be a basic open set in $min (B)$. Let $P \in \varTheta (X_a^A)$. There exists $ Q \in X_a^A$ such that $Q \cap A= P$. There exists a $c \in B$ such that $c \notin Q$ such that $ac \in A$, due to denseness. Our claim is that $P \in X_{ac}^A\subseteq \varTheta (X_a^A)$. As $ a, c \notin Q$ and $Q$ is prime, $ac \notin Q$ and hence $ac \notin P$. So $P \in X_{ac}^A$.

 For the other side, let $Z \in X_{ac}^A$. By theorem \ref{3.6}, there exists unique $T \in min (B)$ such that $T \cap A = Z$. As $ac \notin Z$, $a \notin T$. So $T \in X_a^A$. Hence $\varTheta (T) = Z \in \varTheta (X_a^A)$. Hence $\varTheta (X_a^A)$ is open in $min (A)$

Hence the theorem follows.
\end{proof}

\begin{note} We have however failed to cite any example of weak completely normal ring which is not compleltely normal. We were rather indeed trying to find, whether $A$ is weak completely normal subring with respect to $B$ whenever $A$ is dense in $B$. But we could not crack it and not even traced any counter example so far.
\end{note}

\end{document}